\documentclass{proc-l}

\usepackage{epic, eepic}
\usepackage{units}
\usepackage{amsthm, amsmath, amssymb, amscd}
\usepackage[all]{xy}
\usepackage{mathrsfs}
\usepackage{graphicx}

\copyrightinfo{2011}{This work is hereby placed in the public domain.}

\newtheorem{theorem}{Theorem}[section]
\newtheorem*{faketheorem}{Theorem 1.1}

\newtheorem*{theorem*}{Theorem}
\newtheorem*{problem}{Problem}
\newtheorem*{lemma}{Lemma}

\newtheorem{proposition}{Proposition}

\theoremstyle{definition}

\newtheorem{example}[theorem]{Example}

\theoremstyle{remark}

\numberwithin{equation}{section}

\renewcommand{\b}{\beta}
\def\s{\sigma}
\def\R{\mathbb{R}}
\def\P{\mathbb{P}}
\def\Z{\mathbb{Z}}
\def\D{\Delta}
\def\Pnk{\tilde\Pi_{\le n}^k}
\newcommand{\nf}[2]{\text{\Large $\nicefrac{#1}{#2}$}}
\newcommand{\Stir}[2]{\genfrac{\{}{\}}{0pt}{}{#1}{#2}}

\begin{document}

\title{Discretized configurations and partial partitions}


\author{Aaron Abrams}
\address{Mathematical Sciences Research Institute \\
17 Gauss Way \\
Berkeley CA 94720}
\curraddr{}
\email{abrams.aaron@gmail.com}
\thanks{}

\author{David Gay}
\address{Department of Mathematics \\
University of Georgia \\
Athens GA 30602 }
\curraddr{}
\email{d.gay@euclidlab.org}
\thanks{}

\author{Valerie Hower}
\address{Department of Mathematics \\
University of California, Berkeley \\
Berkeley CA 94720}
\email{vhower@math.berkeley.edu}


\date{September 15, 2010}

\dedicatory{}

\commby{Ken Ono}

\begin{abstract} 
We show that the discretized configuration space of $k$ points in the
$n$-simplex is homotopy equivalent to a wedge of spheres of dimension
$n-k+1$.  This space is homeomorphic to the order complex of the poset of 
ordered partial partitions of $\{1,\dots,n+1\}$ with exactly $k$ parts.  We
compute the exponential generating function for the Euler characteristic of this space
in two different ways, thereby obtaining
a topological proof of a combinatorial recurrence satisfied by the Stirling numbers
of the second kind.
\end{abstract}

\maketitle


\section{Introduction}


\subsection{Configurations}
The discretized configuration space $D_k(X)$ was introduced in \cite{thesis} as 
a combinatorial model of the classical configuration space of $k$-tuples of distinct
points in a space $X$.  (When $k=2$ this has classically been called the ``deleted product''.)
To define $D_k(X)$ it is required that $X$ have the structure of a cell complex.
In \cite{thesis}, and in the works of several subsequent authors, the space $X$ is a graph, 
i.e.~a finite 1-complex.

To study $D_k(X)$ for higher dimensional $X$, it is natural to begin with some basic
building blocks.  In this paper we consider the discretized configuration spaces $D_k(\D^n)$, 
where $\D^n$ is the $n$-dimensional simplex.  Note that if $X$ is any simplicial complex,
then $D_k(X)$ is built out of (products of) spaces of the form $D_i(\D^n)$.
We prove
\begin{theorem}\label{main} 
The space $D_k(\D^n)$ is homotopy equivalent to a wedge of spheres of dimension $n-k+1$.
\end{theorem}

\begin{theorem}\label{euler}
The number of spheres in the wedge is given by the formula
$$\b_{k,n}=\sum\limits_{i=0}^{k}(-1)^{i+k+1} \binom {k}{i+1} i^{n+1}.$$
More concisely, the Euler characteristic $\chi_{k,n}$ of $D_k(\D^n)$ has the
two-variable exponential generating function
given by
$$\sum\limits_{k,n} \chi_{k,n-1} \frac {x^k}{k!} \frac{y^n}{n!} =  e^{x+y-xe^{-y}}.$$
\end{theorem}

For example, $D_2(\D^n)$ is homotopy equivalent (actually homeomorphic; see
Section \ref{sec:defs}) to an $(n-1)$-dimensional sphere, and $D_3(\D^n)$ is homotopy
equivalent to a wedge of $2^{n+1}-3$ spheres of dimension $n-2$.

Note that, by Theorem \ref{main}, the Betti number $\b_{k,n}$ is related to the
Euler characteristic by 
\begin{equation}\label{bettieuler}
\chi_{k,n}=1+(-1)^{n-k+1}\b_{k,n}.
\end{equation}
Thus combinatorial manipulations are sufficient to prove that the two statements
in Theorem \ref{euler} are equivalent.


\subsection{Partitions}
The partition lattice $\Pi_n$ is a classical combinatorial object studied since antiquity.  
It is a poset whose elements are set partitions of $[n]=\{1,\ldots,n\}$ and whose ordering
is given by refinement.  The order complex of $\Pi_n$ has the homotopy type of a wedge of 
spheres (see \cite{wachs}).  The (larger) lattice $\Pi_{\le n}$ of \emph{partial partitions} has 
recently also been shown to have the homotopy type of a wedge of spheres \cite{tricia}.

These results are proved using the machinery of algebraic combinatorics.
The standard technique is to find a shelling of the order complex, often via an 
EL- or CL-labelling of the poset \cite{wachs}.  One nice feature of these arguments
is that they often produce an explicit basis for the (unique) nonzero homology group, 
and one can usually compute its Betti number.


\subsection{Connection} \label{connection}
The face poset of the space $D_k(\D^{n-1})$ can be identified with a
combinatorial object closely related to $\Pi_{\le n}$, namely the poset $\Pnk$ 
of \emph{ordered partial partitions of $[n]$ with exactly $k$ parts}.  Thus 
Theorems \ref{main} and \ref{euler} imply the following.

\begin{theorem}\label{poset}
The poset $\Pnk$ has the homotopy type of a wedge of spheres of dimension $n-k$.
With notation from Theorem \ref{main}, the number of spheres is $\b_{k,n-1}$ and the 
Euler characteristic is $\chi_{k,n-1}$.
\end{theorem}

Note that the symmetric group $S_k$ acts 
on the space $D_k(X)$ (and on the poset $\Pnk$) by permuting coordinates.  The 
quotient space $UD_k(X)$ has a topological interpretation as an 
\emph{unordered discretized configuration space}.  The
quotient poset $\Pi_{\le n}^k$ is naturally a subposet of $\Pi_{\le n}$; its elements are the
\emph{partial partitions of $[n]$ with exactly $k$ parts}.  However, $\Pi_{\le n}^k$
does not have the homotopy type of a wedge of spheres, as can be seen
in the case $k=2$ (where one obtains a real projective plane $\R\P^{n-2}$). 

Several other authors have studied the topology of the partition lattice and 
numerous related posets \cite{Hanlon81, Hanlon03,  Kozlov00, Sundaram94, Sundaram96, 
Sundaram01, SundaramWachs94}.  Many (but not all) of these spaces have the homotopy 
type of a wedge of spheres.  These authors also study the actions of the symmetric group 
on the homology of the posets.  We do not attempt this in the present paper but believe that 
a careful analysis of the symmetric group action on $H_{n-k+1}(D_k(\Delta^n))$ would be 
interesting. 


\subsection{Paper contents}
We start with definitions, examples, and results in Section \ref{sec:defs}.
In Section \ref{sec:posetpf} we show that Theorem \ref{poset} is equivalent
to Theorems \ref{main} and \ref{euler}.
The proof of Theorem \ref{main} is a direct computation using algebraic 
topology.  It takes two steps:  in Section \ref{sec:pi_1}, we give an inductive argument that 
the spaces are simply
connected (when the dimension is at least 2), and then in Section \ref{sec:homology}
a spectral sequence computation shows that the spaces have the same 
homology as a wedge of spheres.  Together, these allow us to use the Whitehead 
theorem to deduce the result.  

In particular, we do not give a shelling of the spaces $D_k(\D^n)$, although we
suspect that one may be possible.

\begin{problem}
Find a combinatorial proof of Theorem \ref{main}.
\end{problem}

In Section \ref{sec:euler} we prove Theorem \ref{euler}
by using the interpretation as a partition lattice to count the $i$-dimensional
cells in $D_k(\D^n)$.  This count involves the Stirling numbers of the second kind, 
denoted $$\Stir N K$$
which (by definition) means the number of partitions of a set of size $N$ into exactly
$K$ nonempty subsets.  The theorem is proved using well-known and elementary facts
about Stirling numbers.

We also obtain a recurrence for the top Betti number by following through
the spectral sequence.  One can prove that the expression for $\b_{k,n}$ in Theorem \ref{euler}
agrees with the actual Betti numbers for small $k$ and satisfies the same recurrence;
this establishes (the first half of) Theorem \ref{euler} in a different way.


\section{Definitions, theorem, examples}\label{sec:defs}

\subsection{Configurations}

Let $k,n\ge 1$ be fixed integers. The $n$-simplex $\D^n$ is the largest simplicial 
complex on the vertex set $[n+1]$.  The space $D_k(\D^n)$ is the largest cell complex 
that is contained in the product $(\D^n)^k$ minus its \emph{diagonal} 
$\{(x_1,\ldots,x_k)\in(\D^n)^k \mid x_i=x_j \mbox{ for some } i\ne j\}$.
One can also describe this as the union of those open cells of $(\D^n)^k$
whose closure misses the diagonal; explicitly,
$$D_k(\D^n)=\bigcup\limits_{
\genfrac{}{}{0pt}{2} 
{\bar\s_i \mbox{ {\scriptsize{pairwise disjoint}}} } 
{\mbox{{\scriptsize{closed cells in }}} \D^n} 
}
\s_1\times\cdots\times\s_k.$$
We visualize this space by imagining $k$ ``robots'' in the space $\D^n$; 
then $D_k(\D^n)$ is the space of allowable configurations if a robot is said to 
``occupy'' the entire closure of the cell in whose interior it is contained, and 
a configuration is ``allowable'' if no two robots occupy the same point of $\D^n$.
The maximum dimension of a cell of $D_k(\D^n)$ is $n-k+1$.

\begin{faketheorem}[again]
For all $n\ge 1$ and $2\le k \le n+1$, the space $D_k(\D^n)$ is homotopy 
equivalent to a wedge of spheres of dimension $n-k+1$.
\end{faketheorem}

\begin{proof}
If $n-k+1=0$, then $D_k(\D^n)$ is a discrete set of $k!\geq 2$ points, hence 
a wedge of 0-spheres.  If $n-k+1=1$, then $D_k(\D^n)$ is a connected 1-complex, hence up
to homotopy, a wedge of circles.  If $n-k+1>1$, then $2\le k < n$, so Proposition 
\ref{pi_1} (Section \ref{sec:pi_1}) shows that $D_k(\D^n)$ is simply connected and
Proposition \ref{homology} (Section \ref{sec:homology}) shows that $D_k(\D^n)$ 
has the homology of a wedge 
of spheres of dimension $n-k+1$.  As there is clearly a map from a wedge
of spheres inducing isomorphisms on homology, the Whitehead theorem
implies the result.
\end{proof}


\subsection{Examples}

Let $K_d$ denote the complete graph on the vertex set $[d]$ (i.e., the 1-skeleton 
of $\D^{d-1}$).  Let $S^d$ denote the $d$-sphere.  
Let $e_i$ denote the $i$th standard basis vector of $\R^d$, for $1\le i\le d$.

\begin{example}[$k=2$]\label{ex:k=2}
We view $\D^n$ as the convex hull of the $n+1$ standard basis vectors $e_i$
in $\R^{n+1}$.  Then the \emph{Gauss map} from $D_2(\D^n)$ to $\R^{n+1}$
given by $(x,y)\mapsto x-y$ is a homeomorphism onto $S$, the boundary of the
polytope with vertex set $\{e_i-e_j \, | \, i,j\in[n+1], i\ne j\}$ in $\R^{n+1}$.  To see this,
note that the map is obviously continuous and surjective, and given a point in $S$
one can determine the coordinates $x$ and $y$ of the preimage as the ``positive'' 
and ``negative'' parts of $S$.  Thus the map is a bijection.  

The complex $S$ is contained in the hyperplane $\{\sum x_i=0\}\cong\R^n$ and 
is homeomorphic to $S^{n-1}$.  This agrees with Theorems \ref{main} and \ref{euler}.
Note also that $D_2(\D^{n-1})$ is contained in 
$D_2(\D^n)$ as an equator.
\end{example}

We next describe some small cases.  

\subsubsection*{$n=2$}
The space $D_2(\D^2)$ is a connected 1-complex with six vertices, each of 
valence two, and six edges. It is a hexagon.  It is the same as $D_2(K_3)$, 
since neither robot may venture into the interior of the 2-cell of $\D^2$.

\subsubsection*{$n=3$}
The space $D_2(\D^3)$ is (the surface of) a cuboctahedron (Figure \ref{cuboct}).  
Deleting the triangular 2-cells leaves the planar 2-complex $D_2(K_4)$.

\begin{figure}[ht]
\begin{center}
\scalebox{.18}{
	\includegraphics{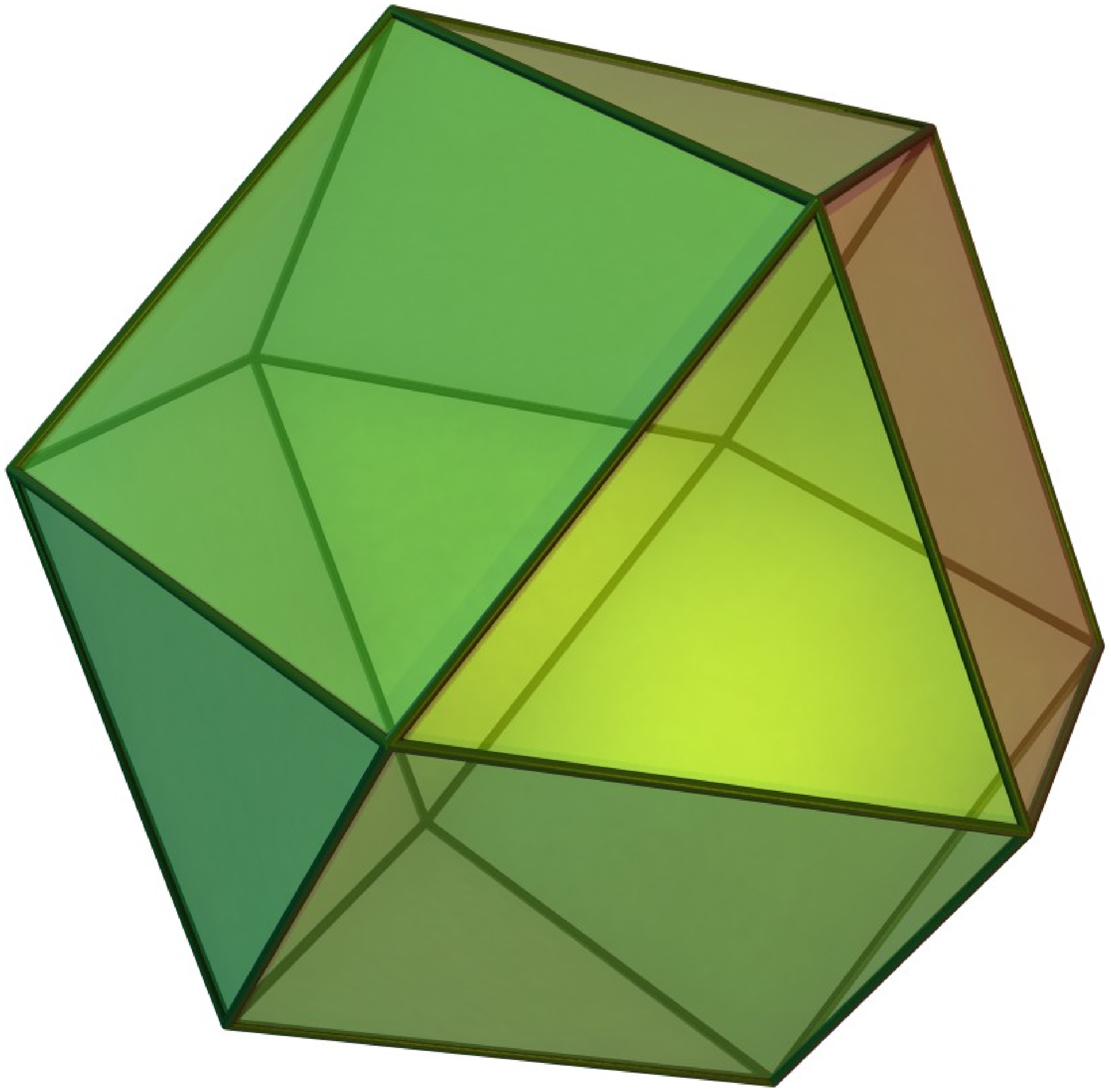}
}
\caption{$D_2(\D^3)$.}
\label{cuboct}
\end{center}
\end{figure}

\subsubsection*{$n=4$}
The space $D_2(\D^4)$ is homeomorphic to $S^3$.  The space $D_2(K_5)$ is a
subcomplex which is itself homeomorphic to a closed orientable surface
of genus six.  This surface is a Heegaard splitting of $D_2(\D^4)$.  The
complementary handlebodies are made of prisms ($\D^2\times I$ and
$I\times\D^2$) and 3-simplices ($v \times\D^3$ and $\D^3\times v$).

\begin{example}[$k=n$]
In this case $D_k(\D^n)$ is a connected 1-complex; a simple count reveals that the
rank of $H_1$ is $\frac 1 2 (n-2)(n+1)! +1$.  It is true but not obvious that this agrees with 
Theorem \ref{euler}.  If one quotients by the action of the
symmetric group $S_k$ that permutes coordinates, the result is the graph $K_{n+1}$
(whose first Betti number is $\binom n 2$).
\end{example}


\subsection{Partitions}\label{sec:posetpf}

An \emph{ordered partition} $\pi$ of the set $[n]=\{1,\ldots,n\}$ is an $r$-tuple (for
some $r$) of disjoint nonempty sets $\pi=(S_1,\ldots,S_r)$ whose union is $[n]$.
(The ordering is on the set of parts, but each part is an unordered set.)
An \emph{ordered partial partition} of $[n]$ is the same, except the union 
is only required to be a subset of $[n]$.  An \emph{ordered partial partition
with exactly $k$ parts} is an ordered partial partition with $r=k$.

Let $\Pnk$ be the poset whose elements are the ordered partial partitions
of $[n]$ with exactly $k$ parts.  The partial order is given by 
$(S_1,\ldots,S_k)\le(T_1,\ldots,T_k)$ iff $S_i\subseteq T_i$ for each $i$.

Note that $\Pnk$ is not a lattice:  writing $(S_1,\ldots,S_k)$ as $(S_i)$, 
the meet $(S_i)\wedge(T_i)$ is $(S_i\cap T_i)$ if
all these sets are nonempty, but it is otherwise nonexistent; and similarly the join 
$(S_i)\vee(T_i)$ is $(S_i\cup T_i)$ if all these sets are disjoint, but otherwise it 
does not exist.  Of course we may add a top and bottom
element if we like.  For comparison, the poset of partial partitions $\Pi_{\le n}$ 
(see \cite{tricia}) has a top element consisting of the 1-part partition $[n]$,
and it is a lattice provided one includes an empty partition at the bottom.

\begin{proof}[Proof of Theorem \ref{poset}]
The face poset of $D_k(\D^{n-1})$ is isomorphic to the poset $\Pnk$,
by mapping the face $\s_1\times\cdots\times\s_k$ to the element
$(S_i)$, where $S_i$ is the set of vertices of the cell $\s_i$.
Thus the order complex of $\Pnk$ is homeomorphic to $D_k(\D^{n-1})$,
so this is equivalent to Theorems \ref{main} and \ref{euler}.
\end{proof}

We remark once again that the quotient $\Pi_{\le n}^k$ of $\Pnk$ by the symmetric 
group $S_k$ does not have the homotopy type of a wedge of spheres.  When $k=2$,
for example, the action is antipodal and the quotient is a projective plane.  Nevertheless
$\Pi_{\le n}^k$ is a subposet of $\Pi_{\le n}$, which is a wedge of spheres up to homotopy
(see the introduction).


\section{The fundamental group}\label{sec:pi_1}

This is the first of the propositions referred to in the proof of Theorem \ref{main}.

\begin{proposition}\label{pi_1}
If $1\leq k < n$, then $D_k(\D^n)$ is simply connected.
\end{proposition}

\begin{proof}
If $k=1$, then $D_k(\D^n)=\D^n$, which is simply connected.
If $k=2$ we have already seen that $D_2(\D^n)$ is homeomorphic to $S^{n-1}$, which
is simply connected if $n>2$.  We proceed by
induction on $k$; let $k>2$ be fixed.  Note that the hypothesis means that if
all robots are at vertices, then there are at least two unoccupied vertices.

To prove the theorem we will construct a set of generators of $\pi_1(D_k(\D^n))$
and then we will show that each is null-homotopic.

Since $D_k(\Delta^n)\subset (\Delta^n)^k$, projection onto the first factor induces 
a map $\rho :D_k(\Delta^n)\longrightarrow \Delta^n$.
The inverse image of a point in the interior of an $i$-cell of $\D^n$ is isomorphic
to $D_{k-1}(\D^{n-i-1})$.  In particular, if $v$ is a vertex of $\D^n$, then $\rho^{-1}(v)$
is simply connected, by induction.


Let $v$ be the vertex $n+1$ of $\D^n$ and let $T$ be the spanning tree of the
1-skeleton of $\D^n$ (that is, $K_{n+1}$) consisting
of all edges incident with $v$.  The space $\rho^{-1}(T)$ is the union of $n+1$
\emph{vertex spaces} (i.e.~the preimages of the vertices), each of which is  
a copy of the simply connected space $D_{k-1}(\D^{n-1})$, and $n$ \emph{edge
spaces} (preimages of edges), each of which is a copy of the connected
(but not necessarily simply connected) space $I\times D_{k-1}(\D^{n-2})$.
The edge spaces are attached to the vertex spaces by embeddings at
the ends $\{0,1\}\times D_{k-1}(\D^{n-2})$.  Thus by the Seifert-van Kampen
theorem, $\rho^{-1}(T)$ is simply connected.

Now consider $Y=\rho^{-1}(K_{n+1})$.  The space $Y$ is obtained from 
$\rho^{-1}(T)$ by attaching $\binom n 2$ edge spaces $I\times D_{k-1}(\D^{n-2})$
indexed by the pairs $i,j\in[n]$ with $i<j$.
As there are no new vertex spaces, each such edge space results in an HNN
extension of the fundamental group; thus $\pi_1(Y)$ is free of rank $\binom n 2$.

Note that the entire 1-skeleton of $D_k(\D^n)$ is contained in $Y$.
Thus a generating set for $\pi_1(Y)$ will also generate $\pi_1(D_k(\D^n))$.
We now describe such a generating set.

Fix a basepoint $\star\in \rho^{-1}(v)$.  For each $i,j\in [n]$ choose a path $\alpha_{ij}$ in 
$\rho^{-1}(v)$ from $\star$ to a configuration $x$ with $i$ and $j$ unoccupied and 
each robot at a vertex of $\D^n$.  Let $\gamma_{ij}$ be the 
loop starting at $x$ that leaves all robots fixed except the first and moves the first robot 
around the triangle $v\to i\to j\to v$.  The loop $\alpha_{ij}\gamma_{ij}\alpha_{ij}^{-1}$ 
represents the generator of $\pi_1(Y)$ arising from attaching the edge space 
$\rho^{-1}([i,j])$.  Letting $i,j$ vary, these $\binom n 2$ loops form a free basis 
for $\pi_1(Y)$.

But clearly each of these generators of $\pi_1(Y)$ is null-homotopic in 
$D_k(\D^n)$, as the loop $\gamma_{ij}$ bounds a 2-simplex in $D_k(\D^n)$.
We conclude that $D_k(\D^n)$ is simply connected, as desired.
\end{proof}


\section{Homology} \label{sec:homology}

Here we prove the second of the propositions referred to in the proof of Theorem \ref{main},
using a spectral sequence to compute the homology of $D_k(\D^n)$.  We refer the reader to \cite{Kozlov01} for a discussion on the use of spectral sequences in combinatorics.  In what follows, all homology groups will have integer coefficients.

Recall or observe:
\begin{enumerate}
\item $D_k(\Delta^n)$ has dimension $n-k+1$. 
\item If $n-k+1>0$, then $D_k(\Delta^n)$ is connected.\label{conn}
\item $D_2(\Delta^n)$ is homeomorphic to a sphere of dimension $n-1$. \label{k=2}
\item $D_k(\Delta^{k-1})$ is $k!$ points.
\end{enumerate}

Let $n\geq1$ and $1\leq k \leq n$ be fixed.

Again we consider the projection $\rho :D_k(\Delta^n)\longrightarrow \Delta^n$ onto
the first coordinate.  Note that $\rho$ satisfies
\begin{eqnarray*}
\rho^{-1}([i_1])&\cong &D_{k-1}(\Delta^{n-1}), \\
\rho^{-1}([i_1,i_2])&\cong &\R\times D_{k-1}(\Delta^{n-2}),  \\
\rho^{-1}([i_1,i_2,i_3])&\cong &\R^2\times D_{k-1}(\Delta^{n-3}),\\
& \vdots &\\
\rho^{-1}([i_1,i_2,i_3, \cdots , i_{n-k+1}])&\cong &\R^{n-k}\times D_{k-1}(\Delta^{k-1}), \\
\rho^{-1}([i_1,i_2,i_3, \cdots , i_{n-k+2}])&\cong &\R^{n-k+1}\times D_{k-1}(\Delta^{k-2}), \\
\end{eqnarray*}
where $i_1,i_2,i_3, \ldots , i_{n-k+2}$ are distinct vertices of $\Delta^n$ and the 
face $[\cdot ]$ of $\Delta^n$ is the interior of the convex hull of 
the given vertices.  We use $\Delta^n(k)$ to denote the $k$-dimensional faces of 
$\Delta^n$ and $\Delta^n_{\leq k}=\bigcup_{i\leq k}\Delta^n(i)$ for the $k$-skeleton 
of $\Delta^n$.  

The map $\rho$ gives a filtration of $D_k(\Delta^n)$ as follows:
$$\emptyset=X_{-1}\subset X_0 \subset X_1 \subset X_2 \subset \cdots 
\subset X_{n-k}\subset X_{n-k+1}=D_k(\Delta^n),$$ 
where $X_p=\rho^{-1}(\Delta^n_{\leq p})$.  Moreover 
\begin{eqnarray*} 
X_p\backslash X_{p-1}&=&\bigsqcup_{f\in \Delta^n(p)}\rho^{-1}(f) \\
& \cong &\bigsqcup_{f\in \Delta^n(p)} \R^p\times D_{k-1}(\Delta^{n-p-1}).   
\end{eqnarray*}
We can hence construct a spectral sequence \cite[p.~327]{M} $(E^r, d^r)$ with 
$$E^1_{p,q}=H_{p+q}(X_p\backslash X_{p-1}) \Longrightarrow H_{p+q}(D_k(\Delta^n))$$ 
converging to homology with closed supports.  Since $H_p(\R^p)=\Z$ 
is the only nonzero homology group of $\R^p$, we have
\begin{eqnarray*}
E^1_{p,q}&=&H_{p+q}(X_p\backslash X_{p-1}) \\
&=&\bigoplus_{f\in \Delta^n(p)}H_{p+q}(\R^p\times D_{k-1}(\Delta^{n-p-1})) \\ 
&=& \bigoplus_{f\in \Delta^n(p)}H_q(D_{k-1}(\Delta^{n-p-1})),
\end{eqnarray*}
where we have used the K\"{u}nneth formula.

\begin{proposition} \label{homology}
Let $n\ge 1$ and $1\le k \le n$.  Then 
\[
H_r(D_k(\Delta^n))
=
\begin{cases}
\Z,& \quad \text{if $r=0$}, \\
0,& \quad \text{if $0<r<n-k+1$},
\end{cases}
\]
and $H_{n-k+1}(D_k(\Delta^n))$ is free abelian and nontrivial.
Thus $D_k(\D^n)$ has the same homology as a wedge
of spheres of dimension $n-k+1$.
\end{proposition}

\begin{proof}
We induct on $k$; the cases $k=1,2$ are observations (\ref{conn}), 
(\ref{k=2}) above.  Assume the theorem holds for configurations of $k-1$ robots.    
Let $(E^r,d^r)$ be the spectral sequence from above.
Using our inductive hypothesis, the $E^1$ term has nonzero entries only along 
the diagonal line $p+q=n-k+1$ and along row $q=0$.   The entries in the $E^1$ 
term are as follows:
$$q \quad \begin{array}{|cccccccc} E^1_{0,n-k+1} \\ 
0& E^1_{1,n-k}\\  
0&0&E^1_{2,n-k-1} \\ 
0&0&0&E^1_{3,n-k-2} \\ 
0 &0&0&0&\ddots\\ 
0&0&0&0&\cdots&\ddots \\
0&0&0&0&\cdots &\cdots &E^1_{n-k,1} \\  
E^1_{0,0} & E^1_{1,0} & E^1_{2,0} & E^1_{3,0} & \cdots & \cdots & E^1_{n-k,0} & E^1_{n-k+1,0} \\ 
\hline \\ 
p=0&1&2&3& \cdots&\cdots &n-k&n-k+1 
\end{array} $$
where 
\begin{align*}
\mathrm{rank }E^1_{p, n-k+1-p}&= \binom{n+1}{p+1} b_{n-k+1-p}(D_{k-1}(\Delta^{n-p-1}))
\mbox{ for }0\leq p \leq n-k, \\
\mathrm{rank }E^1_{p,0}&=
	\begin{cases} 
	\binom{n+1}{p+1} & \mathrm{for}\quad 0\leq p \leq n-k, \\ 
	(k-1)!\binom{n+1}{p+1}& \mathrm{for}\quad p=n-k+1 ,
	\end{cases}
\end{align*}
and all other entries are zero.

The only possible nonzero higher differentials are the horizontal maps 
$d^1_{p,0}:E^1_{p,0}\longrightarrow E^1_{p-1,0}$ for $1 \leq p \leq n-k+1$, 
where
$$d^1_{p,0}:\bigoplus_{f\in\Delta^n(p)} H_0(D_{k-1}(\Delta^{n-p-1})) \longrightarrow 
\bigoplus_{g\in\Delta^n(p-1)} H_0(D_{k-1}(\Delta^{n-p})) $$
is a direct sum of maps 
$$H_0(D_{k-1}(\Delta^{n-p-1}))\longrightarrow 
\bigoplus_{g\in f(p-1)} H_0(D_{k-1}(\Delta^{n-p})) \quad \quad \mbox{for } f\in \Delta^n(p).$$

Note that if $\mathrm{dim}f\leq n-k$ and $g\in f(p-1)$, then the map 
$$\Z\cong H_0(D_{k-1}(\Delta^{n-p-1})) \longrightarrow 
H_0(D_{k-1}(\Delta^{n-p})) \cong \Z$$ 
is injective (hence an isomorphism) as it is induced by inclusion.  Thus, computing 
$E^2_{p,0}=\nf{\mathrm{ker}d^1_{p,0}}{\mathrm{im}d^1_{p+1,0}}$ is equivalent to 
computing the $p$th homology group of $\Delta^n$ for $p <n-k$, and we obtain 
$E^2_{p,0}=0$ for $p <n-k$.

For $p=n-k$, we have that $d^1_{n-k,0}$ is equivalent to $\partial_{n-k}$, where 
$\partial$ is the boundary map for the $n$-simplex.  For each $(n-k+1)$-face 
$f$ of $\Delta^n$, let $\gamma_f$ be a generator for the factor 
of $H_0(D_{k-1}(\Delta^{k-2}))$ corresponding to 
the fiber over $f$.  Then 
$$\left. d^1_{n-k+1,0}\right|_{\bigoplus_{f}\Z\cdot\gamma_f} \cong\partial_{n-k+1}$$ 
and hence 
$\mathrm{im}d^1_{n-k+1,0}\cong\mathrm{im}\partial_{n-k+1}=\mathrm{ker}
\partial_{n-k}\cong\mathrm{ker}d^1_{n-k,0},$ which yields $E^2_{n-k,0}=0$.  
Thus the only nonzero entries of the $E^2$ term are $E^2_{0,0}\cong\Z$ 
and also along the line $p+q=n-k+1$ where we have free abelian groups.  
We hence obtain the integer homology groups of $D_k(\Delta^n)$ from 
$E^2=E^{\infty}$ by adding along the lines $p+q=r$, which proves the inductive step.
\end{proof}


\section{The Euler characteristic}\label{sec:euler}

In this section we discuss two proofs of Theorem \ref{euler}.
We give the first in detail, via the interpretation of $D_k(\D^n)$ as the order complex 
of the poset $\tilde\Pi_{\le n+1}^k$.  The second approach uses the spectral sequence 
computation from Section \ref{sec:homology}; we give an outline and invite the reader
to fill in the details.


\subsection{Stirling numbers}

Recall that the symmetric group $S_k$ acts on $D_k(\D^n)$ by permuting
coordinates.  This is a free cellular action; i.e., the quotient $UD_k(\D^n)$ inherits a
natural cell structure.  An $i$-dimensional cell of $D_k(\D^n)$ corresponds to an ordered 
partial partition of $[n+1]$ which uses exactly $k+i$ of the elements from $[n+1]$, 
and an $i$-dimensional cell of $UD_k(\D^n)$ corresponds to an \emph{unordered} 
partial partition of $[n+1]$ which uses exactly $k+i$ of the elements from $[n+1]$.

The Stirling number of the second kind, denoted $\Stir N K$, is by
definition the number of ways to partition a set of size $N$ into exactly $K$ nonempty 
subsets.  (These partitions are unordered.)  Thus the number of $i$-cells of $UD_k(\D^n)$ is 
$$\binom {n+1}{k+i} \Stir {k+i}{k},$$
and the Euler characteristic of $UD_k(\D^n)$ is 
$$\sum_{i=0}^{n-k+1} (-1)^i \binom {n+1}{k+i} \Stir {k+i}{k}.$$
Note that this sum can be extended to all integer values of $i$, since
the additional terms would all be zero for one reason or another.
As the Euler characteristic is multiplicative under covers, it follows that the Euler
characteristic of $D_k(\D^n)$ is
\begin{equation}\label{first}
\chi_{k,n}=k!\sum_{i} (-1)^i \binom {n+1}{k+i} \Stir {k+i}{k}.
\end{equation}

There are many well-known formulae and recurrences for the Stirling numbers;
see \cite[Sec.~1.6]{gfology} for a nice introduction.  For instance a closed form
is
$$\Stir N K = \frac 1{K!} \sum_{j=0}^K (-1)^{K-j} {\binom K j} j^N.$$

By Equations \eqref{bettieuler} and \eqref{first}, 
the first half of Theorem \ref{euler} follows from the combinatorial identity
$$\sum\limits_{j}(-1)^{n+j-1} \binom {k}{j} (j-1)^{n+1}
=
k!\sum_{i} (-1)^i \binom {n+1}{k+i} \Stir {k+i}{k},$$
which, using the closed form for Stirling numbers, can indeed be proven by 
elementary combinatorial techniques.

\begin{example}[$k=3$]
The theorem implies that $D_3(\D^n)$ is a wedge of
$2^{n+1}-3$ spheres of dimension $n-2$.
\end{example}

\begin{example}[$k=4$]
Similarly, $D_4(\D^n)$ is a wedge of $3^{n+1} - 4 \cdot 2^{n+1} + 6$ spheres
of dimension $n-3$.
\end{example}


\subsection{Generating functions}

The exponential generating function for Stirling numbers is also well-known.
A form that is useful here is that for fixed $k$, 
\begin{equation}\label{second}
\sum\limits_{n=0}^{\infty} \Stir n k \frac {z^n}{n!} = \frac {(e^z -1 )^k}{k!}.
\end{equation}
Plugging in from Equation \eqref{first}, we have
$$\sum\limits_k\sum\limits_n \chi_{k,n-1} \frac {x^k}{k!} \frac {y^n}{n!}
=
\sum\limits_{k,n,i} (-1)^i \binom {n}{k+i} \Stir {k+i}{k}  \frac {x^k y^n}{n!} .$$
Replacing $i$ with the index $j=k+i$ yields
$$\sum\limits_k (-1)^k \left(\sum\limits_{n,j} \binom n j (-1)^j \Stir j k \frac {y^n}{n!} \right) x^k,$$
and we recognize the expression in large parentheses as the product of the 
exponential generating functions for $e^y$ and $\frac 1 {k!}(e^{-y}-1)^k$
(see Equation \eqref{second}).
Thus this sum simplifies to 
$$\sum\limits_k (-1)^k e^y \left(e^{-y}-1\right)^k \frac {x^k}{k!},$$
which in turn equals $e^y e^{x(1-e^{-y})}$.  This establishes 
the second part of Theorem \ref{euler}.

Of course one can differentiate with respect to $y$ to obtain a generating function
with expansion $\sum_{k,n} \chi_{k,n} \frac{x^k}{k!}\frac {y^n}{n!}$.


\subsection{Betti numbers from the spectral sequence}

Another approach to proving the first half of Theorem \ref{euler} is to notice that the
spectral sequence yields a recursion satisfied by the Betti number $\b_{k,n}$.  One
can then show that the formula in the statement of the theorem satisfies the same 
recursion.  We leave the latter part to the reader; again, this can be
carried out by elementary (but nontrivial) combinatorial arguments.

We establish the following recurrence for $\b_{k,n}$ in terms of the values of 
$\b_{k-1,i}$ for $i<n$.

\begin{theorem}\label{recurrence}
For $1\le k \le n+1$, we have 
$$\b_{k,n}=-\binom{n}{k-1}+\sum_{p=0}^{n-k+1} \binom{n+1}{k+p-1}\b_{k-1,k+p-2}.$$
\end{theorem}

It would be nice to have a geometric interpretation of this result.  For instance,
it may be possible to describe bases for the top homology groups in such a way
that geometric relationships between the bases for different values of $k$ and $n$
shed light on this recurrence.

In any case, because $\b_{j,j-1}=j!$, the theorem itself follows immediately from the 
next two lemmas.

\begin{lemma}
For $3\le k \le n$, define $Y_{k,n}:=\mathrm{rank}E^2_{n-k+1,0}$, where $(E^r, d^r)$ 
is the spectral sequence for $D_k(\Delta^n)$.  Then we have 
$$Y_{k,n}=(k-1)! \binom{n+1}{k-1} - \binom{n}{k-1}.$$
\end{lemma}

\begin{proof}
As in Section \ref{sec:homology}, for each $(n-k+1)$-face 
$f$ of $\Delta^n$, let $\gamma_f$ be a generator for the factor 
of $H_0(D_{k-1}(\Delta^{k-2}))$ corresponding to 
the fiber over  $f$.  Then we have the following:
\begin{eqnarray*}
\mathrm{im}d^1_{n-k+1,0}& =&\left. \mathrm{im}d^1_{n-k+1,0}
	\right|_{\bigoplus_{f}\Z\cdot\gamma_f} ,\\
\mathrm{ker}d^1_{n-k+1,0} &=&\underbrace{\Z^{\binom{n+1} {n-k+2}}\oplus \ldots \oplus 
\Z^{\binom{n+1}{n-k+2}} }_{{(k-1)!-1}\mbox{ factors}} \oplus \left. 
	\mathrm{ker}d^1_{n-k+1,0}\right|_{\bigoplus_{f}\Z\cdot\gamma_f} ,
\end{eqnarray*}
where 
$$\left. \mathrm{ker}d^1_{n-k+1,0}\right|_{\bigoplus_{f}\Z\cdot\gamma_f} \cong 
\mathrm{ker}\partial_{n-k+1}= \mathrm{im}\partial_{n-k+2}$$ 
and $\partial$ is the boundary map for the $n$-simplex.
Next, we note that
\begin{eqnarray*}
\mathrm{rank}(\mathrm{im}\partial_{n-k+2}) 
&=&\binom{n+1}{n-k+3} -\mathrm{rank}(\mathrm{ker}\partial_{n-k+2}) \\
&=&\binom{n+1}{n-k+3} -\mathrm{rank}(\mathrm{im}\partial_{n-k+3}).
\end{eqnarray*}
Thus, if we define $f_k:= \mathrm{rank}(\mathrm{im}\partial_{n-k+2})$, then we have 
$f_2=1$ and $f_k = \binom{n+1} {n-k+3} - f_{k-1}$ for $3\le k \le n$.
Iterating, we find that
$f_k=\sum_{j=0}^{k-2}(-1)^{k+j}\binom{n+1} j$.

Thus 
\begin{eqnarray*}
Y_{k,n}&=&\left[(k-1)!-1\right]\binom {n +1} {k-1} + f_k \\
&=&(k-1)!\binom {n +1} {k-1} +\sum_{j=0}^{k-1}(-1)^{k+j}\binom{n+1} j\\
&=&(k-1)!\binom {n +1} {k-1} - \binom {n}{k-1},
\end{eqnarray*}
where in the last line we have used the elementary combinatorial identity
$\sum_{j=0}^K (-1)^j \binom N j = (-1)^K \binom {N-1} K$.
\end{proof}

\begin{lemma}
If $3\le k \le n$, then $\b_{k,n}$ satisfies the following recursion:
$$
\b_{k,n}=Y_{k,n}+ \sum_{p=0}^{n-k}\binom{n+1} {p+1} \b_{k-1,n-p-1},
$$
where $Y_{k,n}$ is as in the above lemma.
\end{lemma}

\begin{proof}
Since the spectral sequence $(E^r,d^r)$ collapses at $E^2$ and there is no 
torsion in the $E^2$-term, we have 
\begin{eqnarray*}
\b_{k,n}&=&\mathrm{rank}E^2_{n-k+1,0} + \sum_{p=0}^{n-k}\mathrm{rank}E^2_{p, n+1-k-p} \\
&=&Y_{k,n} + \sum_{p=0}^{n-k}\mathrm{rank}E^1_{p, n+1-k-p}.
\end{eqnarray*}
Since $\mathrm{rank}E^1_{p, n+1-k-p}= \binom{n+1}{p+1} \b_{k-1,n-p-1}$, 
the lemma holds.
\end{proof}


\bibliographystyle{amsplain}
\bibliography{refs}

\end{document}